\documentclass{amsart}
\usepackage{setspace, amsmath, amsthm, amssymb, amsfonts, amscd, epic, graphicx, ulem, dsfont}
\usepackage[T1]{fontenc}
\usepackage{multirow}
\usepackage{bbm}
\usepackage{enumerate}

\makeatletter \@namedef{subjclassname@2010}{
  \textup{2010} Mathematics Subject Classification}
\makeatother

\newtheorem{thm}{Theorem}[section]
\newtheorem{cor}[thm]{Corollary}
\newtheorem{lem}[thm]{Lemma}
\newtheorem{pro}[thm]{Proposition}
\newtheorem{conj}[thm]{Conjecture}

\theoremstyle{remark}
\newtheorem*{rema}{Remark}

\theoremstyle{definition}

\newcommand{\R}{\mathbb{R}}

\newcommand{\C}{\mathbb{C}}

\begin{document}

\title[Maximality of Operators]{Maximality of Linear Operators}
\author[M. Meziane and M. H. MORTAD]{Mohammed Meziane and Mohammed Hichem Mortad$^*$}

\dedicatory{}
\thanks{* Corresponding author.}
\date{}
\keywords{Normal, self-adjoint and symmetric Operators.
Commutativity. Maximality of operators}

\subjclass[2010]{Primary 47A05,
 Secondary 47A10, 47B20, 47B25}

 \address{ Department of
Mathematics, University of Oran 1, Ahmed Ben Bella, B.P. 1524, El
Menouar, Oran 31000, Algeria.\newline {\bf Mailing address}:
\newline Pr Mohammed Hichem Mortad \newline BP 7085 Seddikia Oran
\newline 31013 \newline Algeria}

\email{m.meziane13@yahoo.fr} \email{mhmortad@gmail.com,
mortad@univ-oran.dz.}

\begin{abstract}
We present maximality results in the setting of non necessarily
bounded operators. In particular, we discuss and establish results
showing when the "inclusion" between operators becomes a full
equality.
\end{abstract}

\maketitle

\section{Introduction}

In the theory of non necessarily bounded linear operators on a
complex Hilbert space $H$,
 we say that an operator $T$ with domain $D(T)\subset H$ is an extension of
$S$ with domain $D(S)\subset H$ when: $D(S)\subset D(T)$ and $Sx=Tx$
for all $x\in D(S)$. We then write $S\subset T$. We say that $S$ is
closed if it possess a closed graph in $H\oplus H$.

The product of $S$ and $T$ is defined
\[(ST)x=S(Tx)\]
for each $x$ on the natural domain
\[D(ST)=\{x\in D(T):~Tx\in D(S)\}.\]

We say that $T$ is invertible if there exists an $S\in B(H)$ (we
then write $T^{-1}=S$) such that
\[ST\subset I\text{ and } TS=I\]
where $I$ is the identity operator on $H$. It is known that the
product $ST$ is closed if for instance $S$ is closed and $T\in
B(H)$, or if $S^{-1}\in B(H)$ and $T$ is closed.

We also recall that an operator $S$ is said to be densely defined if
its domain $D(S)$ is dense in $H$. It is known that in such case its
adjoint $S^*$ exists and is unique. Notice that if $S$, $T$ and $ST$
are all densely defined, then we are only sure of
\[T^*S^*\subset (ST)^*,\]
and a full equality occurring if e.g. $T^{-1}\in B(H)$ or $S\in
B(H)$.

A densely defined operator $S$ is said to be symmetric if $S\subset
S^*$. It is called self-adjoint if $S=S^*$. We say that $S$ is
normal if $S$ is densely defined,  closed and $SS^*=S^*S$. Recall
that the previous is equivalent to $\|Sx\|=\|S^*x\|$ for all $x\in
D(S)=D(S^*)$. We say that $S$ is formally normal if
$\|Sx\|=\|S^*x\|$ for all $x\in D(S)\subset D(S^*)$. Other classes
of operators are defined in the usual fashion. Let us also agree
that any operator is linear and non necessarily bounded unless we
specify that it belongs to $B(H)$. We also assume the basic theory
of operators (see e.g. \cite{Con} or \cite{WEI}). We do recall the
celebrated Fuglede-Putnam Theorem though:

\begin{thm}(for a proof, see e.g. \cite{Con})
Let $T\in B(H)$ and let $M,N$ be two normal non necessarily bounded
operators. Then
\[TN\subset MT\Longrightarrow TN^*\subset M^*T.\]
\end{thm}

One of the main aims of this work is to seek conditions which
transform $S\subset T$ into $S=T$ (which we call a maximality
condition) for some classes of operators, and also in the case of a
product of two operators. This type of results is a powerful tool
when proving results on unbounded operators. For instance, Statement
(3) of the next theorem is used in the proof of the "unbounded"
version of the spectral theorem of normal operators (see e.g.
\cite{RUD}). For other uses, see e.g.
\cite{Gustafson-Mortad-II-2016-JOT} or \cite{Mortad-2015-RCSP}.

Let us now list some known (see e.g. \cite{RUD} or
\cite{SCHMUDG-book-2012}) maximality results:

\begin{thm}\label{thm devinat mortad one condition} Let $S,T$ be two operators with (dense when necessary) domains $D(S)$ and
$D(T)$ respectively such that $S\subset T$. Then $S=T$ when one of
the following occurs:
\begin{enumerate}
  \item $S$ is surjective and $T$ is injective.
  \item $T$ is symmetric and $S$ is self-adjoint (resp. normal). We then say that
  self-adjoint (resp. normal) operators are maximally symmetric.
  \item $T$ and $S$ are normal (we say that
  normal operators are maximally normal). Hence, self-adjoint (resp.
  normal) operators are maximally normal (resp. self-adjoint).
  \item $S$ is normal and $T$ is formally normal.
\end{enumerate}
\end{thm}

In fact, Statements (2) to (4) of the preceding result are all
simple consequences of the following readily verified result:

\begin{pro}Let $S$ and $T$ be two operators of domains $D(S)$ and
$D(T)$ respectively. If $S$ is densely defined and $D(S^*)\subset
D(S)$, then
\[S\subset T\Longrightarrow S=T\]
whenever $D(T)\subset D(T^*)$.
\end{pro}

Let us now say a few words about "double maximality". A known
property (Theorem 5.31, \cite{WEI}) states that if $S$ is a
symmetric operator such that $S\subset R$ and $S\subset T$ where
$R,T$ are self-adjoint and $D(R)\subset D(T)$, then $T=R$. Observe
that the assumption $S$ symmetric is tacitly assumed in $S\subset R$
so there was no need to assume it. What is more, is that the
assumption $S$ being symmetric is not used in the proof of the
previous result. So, we restate this result as (cf. Proposition
\ref{maximality S T T* propo S.A.}):

\begin{pro}Let $S$ be a densely defined operator such that $S\subset R$ and $S\subset T$
where $R,T$ are both self-adjoint. If $D(R)\subset D(T)$, then
$T=R$.
\end{pro}

Closely related to what has just been said, we have:

\begin{pro}\label{maximality double S,T,R}(see
\cite{Mortad-2015-RCSP}, cf.
\cite{Stochel-Szafraniec-2003-JMSJ-domination-commutativity}) Let
$R,S,T$ be three densely defined operators on a Hilbert space $H$
with respective domains $D(R)$, $D(S)$ and $D(T)$. Assume that
\[\left\{\begin{array}{c}
                 T\subset R, \\
                 T\subset S.
               \end{array}
\right.\] Assume further that $R$ and $S$ are self-adjoint. Let
$D\subset D(T)$ ($\subset D(R)\cap D(S)$) be dense. Let $D$ be a
core, for instance, for $S$. Then $R=S$.

\end{pro}

Finally, we recall results on the case when we have a product on one
side of the "inclusion":

\begin{thm}\label{maximality product mains RECALL THM} Let $R,S,T, A,B,C$ be operators such that $T\subset RS$ and $AB\subset C$.
Then:
\begin{enumerate}
  \item $T=RS$ if all $R,S,T$ are self-adjoint (see
  \cite{DevNussbaum-von-Neumann}).
  \item $T=\overline{RS}$ if $R, S, T$ are self-adjoint and $T_0\subset RS$ instead of $T\subset RS$ where $T_0$ is the restriction of $T$ to some domain $D_0(T)$ (see  \cite{Nussbaum-TAMS-commu-unbounded-normal-1969}).
  \item $AB=C$ when $A$ and $B$ are self-adjoint, $B$ is positive
  and $B^{-1}\in B(H)$ and $C$ is normal (\cite{mortad-commutatvity-devinatz-2013}).
  \item $C=BA$ whenever $A,B$ are self-adjoint and $B^{-1}\in
  B(H)$ and $C$ is closed and symmetric  (\cite{Nussbaum-TAMS-commu-unbounded-normal-1969}).
\end{enumerate}
\end{thm}

\begin{rema}
As observed in \cite{DevNussbaum}, the first statement in the
previous theorem does not extend to normal operators. Indeed, just
in the naive case of unitary operators, we have that a product of
\textit{any} two unitary operators is always unitary even when the
two factors of the product do not commute. This observation
motivates the investigation in the case where one operator is
normal.
\end{rema}

\begin{rema}
Another natural question may pop up. In
\cite{DevNussbaum-von-Neumann}, the authors before showing that
$T=RS$, they first showed that $R$ and $S$ commute strongly (i.e.
the corresponding spectral measures commute). So what if we have
$T\subset ABC$, do we still expect $T=ABC$ when all of $T,A,B,C$ are
self-adjoint? The answer is negative (at least as far as the idea of
their proof is concerned). Indeed, we can have a self-adjoint
product of three self-adjoint operators which do not commute
pairwise. In $\R^2$, consider the following self-adjoint matrices:
\[A=\left(
      \begin{array}{cc}
        1 & 1 \\
        1 & 0 \\
      \end{array}
    \right),~B=\left(
                 \begin{array}{cc}
                   0 & 1 \\
                   1 & 2 \\
                 \end{array}
               \right)\text{ and } C=\left(
                                       \begin{array}{cc}
                                         0 & 1 \\
                                         1 & 0 \\
                                       \end{array}
                                     \right).
\]
Then
\[ABC=\left(
                                       \begin{array}{cc}
                                         3 & 1 \\
                                         1 & 0 \\
                                       \end{array}
                                     \right)\]
is self-adjoint. Nevertheless, none of the products $AB$, $AC$ and
$BC$ is self-adjoint, that is,
\[AB\neq BA,~AC\neq CA \text{ and } BC\neq CB.\]
\end{rema}

\section{Some Results on Normality}

The normality of unbounded products of normal operators has been
studied recently. See e.g. \cite{Gustafson-Mortad-2014} and the
references therein. We recall

\begin{lem}\label{mortad normality AB BA O&M Lemma}(\cite{mortad-commutatvity-devinatz-2013}, cf. \cite{DevNussbaum})
Let $A,B$ be normal operators with $B^{-1}\in B(H)$. If $AB\subset
BA$, then $\overline{AB}$ and $BA$ are both normal.
\end{lem}

The chosen idea of proof of the following result is via the
Fuglede-Putnam Theorem (for a different proof, we may proceed as in
\cite{Bong-Mortad-Stochel}).

\begin{thm}\label{normality AB BA bar commu THM}
Let $A,B$ be normal operators with $B\in B(H)$. If $BA\subset AB$,
then $AB$ and $\overline{BA}$ are both normal (and so
$AB=\overline{BA}$).
\end{thm}

\begin{proof}Since $BA\subset AB$, Fuglede Theorem yields
$BA^*\subset A^*B$. Hence (since also $AB$ is densely defined),
\[B^*BAA^*\subset B^*ABA^*\subset B^*AA^*B=B^*A^*AB\subset (AB)^*AB.\]
Since $AB$ is closed, it follows that $(AB)^*AB$ is self-adjoint,
and by the boundedness of $B^*B$, we get
\[(AB)^*AB\subset AA^*B^*B \text{ or merely } (AB)^*AB=AA^*B^*B=AA^*BB^*\]
by Theorem \ref{maximality product mains RECALL THM}. Similarly, we
obtain
\[AB(AB)^*=AA^*BB^*,\]
and this marks the end of the proof of the normality of $AB$.

To show that $\overline{BA}$ is normal, we first observe that
\[\overline{BA}=(BA)^{**}=(A^*B^*)^*.\]
Now, since $BA\subset AB$, clearly $B^*A^*\subset A^*B^*$. The first
part of the proof leads to the normality of $A^*B^*$ because both
$A^*$ and $B^*$ are normal. Accordingly, $(A^*B^*)^*$ too is normal,
that is, $\overline{BA}$ is normal.
\end{proof}

The following result is known to most readers (a proof based on the
spectral theorem may be found in \cite{Bong-Mortad-Stochel}). We can
equally regard it as a consequence of the preceding theorem:

\begin{cor}\label{Corollary product self-adjoint one bd}
Let $A,B$ be self-adjoint operators with $B\in B(H)$. If $BA\subset
AB$, then $AB$ and $\overline{BA}$ are both self-adjoint.
\end{cor}

\begin{proof}Since $BA\subset AB$, and $A$ and $B$ are self-adjoint,
the previous theorem yields the normality of $\overline{BA}$. But
\[BA\subset AB\Longrightarrow \overline{BA}\subset AB=(BA)^*,\]
i.e. $\overline{BA}$ is symmetric as well. Therefore,
$\overline{BA}$ is self-adjoint. Accordingly,
\[AB=(BA)^*=\overline{BA},\]
and so $AB$ is also self-adjoint, as required.
\end{proof}

\section{Main Results on Maximality}

The same idea of proof of (Theorem 5.31, \cite{WEI}, discussed
above) may lead to the following result which seems to have escaped
notice up to now.

\begin{pro}\label{maximality S T T* propo S.A.}Let $S$ be a densely defined operator such that $S\subset T$ and $S\subset
T^*$. If $D(T)=D(T^*)$, then $T$ is self-adjoint.
\end{pro}

\begin{proof}For all $x\in D(T)=D(T^*)$ and for all $y\in
D(S)\subset D(T)=D(T^*)$ we may write
\begin{align*}
<Tx,y>=&<x,T^*y>\\
=&<x,Sy>\\
=&<x,Ty>\\
=&<T^*x,y>.
\end{align*}
Since $D(S)$ is dense, it follows that $Tx=T^*x$ for all $x\in
D(T)=D(T^*)$, that is, $T$ is self-adjoint.
\end{proof}

\begin{cor}Let $S$ be a densely defined operator such that $S\subset T$ and $S\subset
T^*$. If $T$ is normal, then it is self-adjoint.
\end{cor}

The next result is perhaps known:

\begin{pro}Let $A,B$ be two linear operators on a Hilbert space $H$. Assume
also that $B\in B(H)$. Assume further that $A$ has a domain $D(A)$
and that $A\subset B$.
\begin{enumerate}
  \item We do not necessarily have $A=B$ if $A$ is densely defined but not
  closed.
  \item We do not necessarily have $A=B$ if $A$ is closed but not densely
  defined.
  \item Assume now that $A$ is closed. Then
  \[A=B\Longleftrightarrow \overline{D(A)}=H.\]
  Particularly, if $C$ is invertible, then
  \[AC\subset B\Longrightarrow AC=B.\]
\end{enumerate}

\end{pro}

\begin{proof}First, remember that $A\subset B$ means that $Ax=Bx$ for
all $x\in D(A)$, i.e. $A$ is bounded on $D(A)$.

\begin{enumerate}
  \item We only have $B=\overline{A}$. Since $A$ is densely defined, from $A\subset
  B$, we get that $B^*\subset A^*$. But $D(B^*)=H$ and so $B^*=A^*$.
  Hence
  \[B=A^{**}=\overline{A}.\]

  For a counterexample, just consider $A=B|D$ ($B$ restricted to
  $D$) where $D$ is dense in $H$ but not closed. Since $D$ is not closed,
  $A$, which is bounded on $D$, cannot be closed. Observe in the end that $A\neq B$
  because $D\neq H$!
  \item Just consider $A=0$ (the zero operator) on the
  trivial domain $D(A)=\{0\}$. Take $B$ to be any non-zero bounded operator. Since $A(0)=0=B(0)$, we see plainly that
  $A\subset B$. Finally, it is clear that $A$ is closed on $D(A)$,
  that $D(A)$ is not dense in $H$ and that $A\neq B$.
  \item The implication "$\Rightarrow$" is evident. One way
  of proving the reverse implication is as follows:
 As mentioned above, $A$ is bounded on $D(A)$. Since $A$ is
    closed, $D(A)$ is closed. By hypothesis, $\overline{D(A)}=H$ and
    so $D(A)=H$. This leads to $A=B$.

  Finally, observe that as $AC\subset B$ and $C$ is invertible, we then get that $A\subset BC^{-1}$. By the
  first part of this answer and since $BC^{-1}\in B(H)$, we obtain
  $A=BC^{-1}$. Thus,
  \[AC=BC^{-1}C=B,\]
  as required.
\end{enumerate}
\end{proof}

Closely related to the foregoing theorem, we have:

\begin{lem}\label{Lemma basic stochel}
Assume that $S$ is closed and densely defined in $H$, $B\in B(H)$ is
self-adjoint and $SB\subset I$. Then $B$ is injective, $M=D(SB)$ is
closed and $SB=I_M$.
\end{lem}

\begin{proof}
Since $S$ is closed and $B\in B(H)$, the general theory says that
$SB$ is closed. This combined with $SB\subset I$ completes the
proof.
\end{proof}

\begin{pro}\label{Stochel Proposition}
Assume that $B\in B(H)$ is injective and self-adjoint, and $B^{-1}$
is not bounded. Then there exists a closed, densely defined and
injective operator $S$ in $H$ such that $SB\subset I$ and $SB$ is
not densely defined.
\end{pro}

\begin{proof}
Since, by assumption, $A:=B^{-1}$ is self-adjoint and unbounded, we
see that $D(A^2)\subsetneq D(A)$ (by applying Lemma A.1 in
\cite{SEbestyen-Stochel} to $R=|B|$)). Then, take a (necessarily
nonzero) vector $e\in D(A)\setminus D(A^2)$. It follows from Lemma
3.2 of \cite{SEbestyen-Stochel}, that $M:=D(A)\ominus_A <e>$ is a
vector subspace of $D(A)$ which is dense in $H$, where $\ominus$
designates the orthogonal difference with respect to the graph inner
product of $A$ (cf. \cite{SEbestyen-Stochel}) and $<e>=\C\cdot e$.
Set $S=A|_M$. Since $M$ is a closed vector subspace of $D(A)$ with
respect to the graph norm of $A$, we see that the operator $S$ is
closed, densely defined and injective. Then clearly
\[SB=\left(B^{-1}|_M\right)B\subset B^{-1}B=I\]
and, because $A$ is injective and $D(A)\ominus_A <e>\neq D(A)$, we
have
\[D(SB)=B^{-1}(D(S))==A(D(A)\ominus_A <e>)\subsetneq A(D(A))\subset H.\]
Since, by Lemma \ref{Lemma basic stochel}, $D(SB)$ is closed, we are
done.
\end{proof}

The following gives more information about Theorem \ref{maximality
product mains RECALL THM} is:

\begin{thm}\label{123}Let $A,B,T$ be non necessarily bounded operators such that $A$ is self-adjoint,
$B$ is symmetric with $B^{-1}\in B(H)$ (hence $B$ is self-adjoint)
and $T$ is symmetric. Assume further that $AB\subset T$. Then:
  \begin{enumerate}
    \item $AB\subset BA$.
    \item $BA$ is normal.
    \item $\overline{T}=(BA)^*$.
    \item $T$ is essentially self-adjoint.
  \end{enumerate}
If $T$ is also closed, then $BA$ is self-adjoint and
\[T=BA\text{ and } T=\overline{AB}.\]
\end{thm}

\begin{proof}\hfill
\begin{itemize}
  \item Since $T$ is densely defined, so is $AB$ and so
\[T^*\subset(AB)^*=B^*A^*=BA\]
since also $B^{-1}\in B(H)$ and $A$ and $B$ are self-adjoint. Since
$T$ is symmetric, we obtain
\[AB\subset T\subset T^*\subset BA.\]

Lemma \ref{mortad normality AB BA O&M Lemma} (or else) then yields
the normality of $BA$.

Now, since $T^*\subset BA$, we get $(BA)^*\subset
T^{**}=\overline{T}$. Because $BA$ is normal, so is $(BA)^*$. But,
normal operators are maximally symmetric. Therefore, we finally
infer that
\[(BA)^*=\overline{T},\]
i.e. $T$ is essentially self-adjoint (for $\overline{T}$ is normal
and symmetric).
  \item Suppose now that $T$ is also closed. From above, it is self-adjoint
  and $(BA)^*=T$. Hence
  \[T=(BA)^*=(BA)^{**}=\overline{BA}=BA\]
  since $BA$ is closed.

  In fine,
  \[\overline{AB}=(AB)^{**}=(BA)^*=T.\]
\end{itemize}
\end{proof}

\begin{cor}
Let $A,B,T$ be non necessarily bounded operators such that $A$ is
self-adjoint, $B$ is symmetric with $B^{-1}\in B(H)$ (hence $B$ is
self-adjoint) and $T$ is symmetric. Assume further that $AB\subset
T$. Then
\[A=BAB^{-1}.\]
\end{cor}

\begin{proof}
From Theorem \ref{123}, we have $AB\subset BA$. Left and right
multiplying by $B^{-1}$ give
\[B^{-1}A\subset AB^{-1}.\]
Since $B^{-1}\in B(H)$, Corollary \ref{Corollary product
self-adjoint one bd} yields the self-adjointness of $AB^{-1}$. We
may also write
\[AB\subset BA\Longrightarrow A\subset B(AB^{-1}).\]
Finally, Theorem \ref{maximality product mains RECALL THM} yields
\[A=BAB^{-1},\]
finishing the proof.
\end{proof}

\begin{rema}
In general,
\[BA\subset T\not\Longrightarrow BA=T\]
even when $A$, $B$ and $T$ are all self-adjoint. Indeed, just
consider an invertible self-adjoint $A$ with a domain
$D(A)\subsetneq H$ such that $A^{-1}=B\in B(H)$ and $T=I_H$ (the
identity operator on the whole space $H$). Then
\[BA=A^{-1}A=I_{D(A)}\subsetneq I_H=T\]
where $I_{D(A)}$ is the identity operator on $D(A)$.
\end{rema}

We also have:

\begin{thm}\label{GGHH}
Let $A,B$ and $T$ be operators where $B\in B(H)$. If
 $T^*$ is symmetric, $B$ is self-adjoint and $A$ is
normal, then
\[T\subset AB \Longrightarrow \overline{T}=AB.\]
In particular, if we further assume that $T$ is closed, then we
obtain $T=AB$.
\end{thm}

\begin{proof}
Clearly,
\[T\subset AB\Longrightarrow \overline{T}\subset AB.\]
Hence
\[T\subset AB \Longrightarrow BA^*\subset(AB)^*\subset T^*\subset T^{**}=\overline{T}\subset AB.\]
The Fugelde-Putnam Theorem then gives
\[BA\subset A^*B.\]

Reasoning as in the proof of Theorem \ref{normality AB BA bar commu
THM}, we may prove
\[(AB)^*AB=AB(AB)^*~(=AA^*B^2),\]
i.e. $AB$ is normal. Hence $(AB)^*$ too is normal. Since normal
operators are maximally symmetric, we get
\[(AB)^*\subset T^*\Longrightarrow (AB)^*=T^*\Longrightarrow AB=\overline{AB}=(AB)^{**}=T^{**}=\overline{T}.\]
\end{proof}

\begin{cor}
Let $A,B$ and $T$ be operators where $B\in B(H)$. If
 $T$ is symmetric, $B$ is self-adjoint and $A$ is
normal, then
\[BA\subset T\Longrightarrow \overline{T}=\overline{BA}.\]
\end{cor}

\begin{proof}As above, we get
\[\overline{BA}(BA)^*=(BA)^*\overline{BA}~(=A^*AB^2).\]
Since normal operators are maximally symmetric, we obtain
\[BA\subset T\Longrightarrow \overline{BA}\subset \overline{T} \Longrightarrow \overline{T}=\overline{BA},\]
as needed.
\end{proof}

From the proof of Theorem \ref{GGHH}, we have:

\begin{cor}
Let $A,B$ and $T$ be operators where $B\in B(H)$. If
 $T$ is symmetric, $B$ is self-adjoint and $A$ is
normal, then
\[T\subset AB \Longrightarrow \overline{BA}=A^*B.\]
\end{cor}

\begin{proof}We have already obtained:
\[BA\subset A^*B \text{ and } BA^*\subset AB.\]
These two inequalities allow us to establish the normality of both
$\overline{BA}$ and $A^*B$ (cf. Theorem \ref{normality AB BA bar
commu THM}). Therefore,
\[\overline{BA}=A^*B\]
\end{proof}

\begin{cor}Let $A,B,T$ be non necessarily bounded operators such that $A$ is self-adjoint,
$B$ is symmetric with $B^{-1}\in B(H)$  and $T$ is normal. Then:
\[AB\subset T\Longrightarrow A=TB^{-1}.\]
\end{cor}

\begin{proof}Obviously,
\[AB\subset T\Longrightarrow A\subset TB^{-1}\Longrightarrow B^{-1}T^*\subset A\subset TB^{-1}\Longrightarrow B^{-1}T\subset T^*B^{-1}\]
where we used the Fuglede-Putnam Theorem in the lase implication. As
in the preceding corollary, we may show the normality of $TB^{-1}$.
This, combined with the self-adjointness of $A$ and $A\subset
TB^{-1}$ lead finally to $A=TB^{-1}$, as needed.
\end{proof}

\begin{rema}
We already observed in the remark just above Theorem \ref{GGHH} that
if $A$, $B$ and $T$ are as in the previous corollary, then we must
not have $T=AB$. The same counterexample may be reused here.
\end{rema}

The following is also worth stating.

\begin{cor}\label{456} Let $A,B,T$ be operators such that $A$ is normal, $B$ is
bounded and self-adjoint and $T$ is self-adjoint. Then
\[T\subset AB\Longrightarrow T=AB.\]
\end{cor}

\begin{proof}
As in the proofs above, we can easily show that $AB$ is normal. Then
Theorem \ref{thm devinat mortad one condition} does the remaining
job.
\end{proof}

\begin{thm}
Let $A,B$ and $T$ be non-necessarily bounded operators. Assume that
$B$ is normal, that $A$ is symmetric and invertible (hence $A$ is
self-adjoint) and that $T$ is self-adjoint. Then
\[T\subset AB\Longrightarrow T=AB.\]
\end{thm}

\begin{proof}
We claim that $AB$ is normal. First we have:
\begin{align*}
T\subset AB &\Longrightarrow B^*A\subset T\subset AB\\
&\Longrightarrow A^{-1}B^*AA^{-1}\subset A^{-1}ABA^{-1}\\
&\Longrightarrow A^{-1}B^*\subset BA^{-1}\\
&\Longrightarrow A^{-1}B\subset B^*A^{-1}\text{ (by Fuglede-Putnam Theorem)}\\
&\Longrightarrow BA\subset AB^*.
\end{align*}

Hence
\[(AB)^*AB\supset B^*BA^2 \text{ or } (AB)^*AB\subset A^2B^*B\]
as $(AB)^*AB$ is self-adjoint since $AB$ is closed because also
$A^{-1}\in B(H)$ . Therefore,
\[(AB)^*AB= A^2B^*B\]
by Theorem \ref{maximality product mains RECALL THM}. Similarly, we
may prove that
\[AB(AB)^*=A^2B^*B.\]
Accordingly, $AB$ is normal. In the end, since self-adjoint
operators are maximally normal, we obtain
\[T\subset AB\Longrightarrow T=AB,\]
as required.
\end{proof}

\section{A Conjecture}

Unfortunately, if we switch the roles of $A$ and $B$ in Corollary
\ref{456}, then we have not been able so far to find a complete
answer. Indeed, we need a version of Fuglede-Putnam Theorem which is
not available in the literature yet. Even with help from Bent
Fuglede himself, we have only got as far as the following (we have
chosen not to include the proof in this paper):

\begin{thm}\label{Fuglede NEW Theorem 2017}
Let $B$ be a bounded normal operator with a (finite) pure point
spectrum and let $A$ be a closed (possibly unbounded) operator on a
separable complex Hilbert space $H$. Let $f,g:\C\to \C$ be two
continuous functions. Then
\[BA\subset Af(B)\Longrightarrow g(B)A\subset A(g\circ f)(B).\]
\end{thm}

\begin{cor}\label{Fugelde partial coro}With $A$ and $B$ as above, we have
\[BA\subset AB^*\Longrightarrow B^*A\subset AB.\]

\end{cor}

\begin{proof}
Just apply Theorem \ref{Fuglede NEW Theorem 2017} to the functions
$f,g:z\mapsto \overline{z}$ (so that $g\circ f$ becomes the identity
map on $\C$).
\end{proof}

\begin{cor}With $A$ and $B$ as above, we have
\[T\subset AB\Longrightarrow T=AB\]
if we also suppose that $A$ and $T$ are self-adjoint.
\end{cor}

\begin{proof}
Apply Corollary \ref{Fugelde partial coro}...
\end{proof}

Related to what has just been discussed, we propose the following
conjecture:

\begin{conj}
Let $A$ be an operator (densely defined and closed if necessary) and
let $B\in B(H)$ be normal. Then
\[BA\subset AB^* \Longrightarrow B^*A\subset AB.\]
\end{conj}

\begin{rema}
What makes the previous conjecture interesting is that it is known
to hold if $A\in B(H)$ (Fuglede-Putnam Theorem), and as it is posed,
it is covered by none of the known (unbounded) generalizations of
Fuglede-Putnam Theorem (see e.g.
\cite{Mortad-Fuglede-Putnam-CAOT-2012},
\cite{Paliagiannis-2015-Fuglede-Newest} and \cite{STO}).
\end{rema}

\section{Acknowledgments}
The authors would like to thank Professor Jan Stochel for
Proposition \ref{Stochel Proposition} as well as Professor Bent
Fuglede for Theorem \ref{Fuglede NEW Theorem 2017}. Both results
were communicated to the corresponding author via email.


\begin{thebibliography}{1}

\bibitem{Bong-Mortad-Stochel}
Il Bong Jung, M. H. Mortad, J. Stochel, \textit{On normal products
of selfadjoint operators}, Kyungpook Math. J., \textbf{57} (2017)
457-471.

\bibitem{Con}
J. B. Conway, \textit{A course in functional analysis}, \textnormal{
Springer, 1990 (2nd edition)}.

\bibitem{DevNussbaum-von-Neumann}
A. Devinatz, A. E. Nussbaum, J. von Neumann, \textit{On the
Permutability of Self-adjoint Operators}, Ann. of Math. (2), {\bf
62} (1955) 199-203.

\bibitem{DevNussbaum}
A. Devinatz, A. E. Nussbaum, \textit{On the Permutability of Normal
Operators}, Ann. of Math. (2), {\bf 65} (1957) 144-152.

\bibitem{Gustafson-Mortad-2014}
K. Gustafson, M. H. Mortad, \textit{Unbounded Products of Operators
and Connections to Dirac-Type Operators},  Bull. Sci. Math.,
\textbf{138/5}  (2014) 626-642.

\bibitem{Gustafson-Mortad-II-2016-JOT}
K. Gustafson, M. H. Mortad, \textit{Conditions Implying
Commutativity of Unbounded Self-adjoint Operators and Related
Topics}, J. Operator Theory, \textbf{76/1} (2016) 159-169.

\bibitem{Mortad-Fuglede-Putnam-CAOT-2012}
M. H. Mortad, \textit{An All-Unbounded-Operator Version of the
Fuglede-Putnam Theorem},  Complex Anal. Oper. Theory,  {\bf 6/6}
(2012), 1269-1273. 

\bibitem{mortad-commutatvity-devinatz-2013}
M. H. Mortad, Commutativity of Unbounded Normal and Self-adjoint
Operators and Applications, \textit{Operators and Matrices},
\textbf{8/2} (2014), 563-571.

\bibitem{Mortad-2015-RCSP}
M.H. Mortad, A criterion for the normality of unbounded operators
and applications to self-adjointness, \textit{Rend. Circ. Mat.
Palermo (2)}, \textbf{64/1} (2015) 149-156

\bibitem{Nussbaum-TAMS-commu-unbounded-normal-1969}
A. E. Nussbaum, \textit{A Commutativity Theorem for Unbounded
Operators in Hilbert Space}, Trans. Amer. Math. Soc., {\bf 140}
(1969) 485-491.

\bibitem{Paliagiannis-2015-Fuglede-Newest}
F. C. Paliogiannis, \textit{A generalization of the Fuglede-Putnam
theorem to unbounded operators}, J. Oper., (2015). Art. ID 804353, 3
pp.

\bibitem{RUD}
W. Rudin, Functional Analysis, \textit{McGraw-Hill Book Co.}, Second
edition, International Series in Pure and Applied Mathematics,
McGraw-Hill, Inc., New York, 1991.

\bibitem{SCHMUDG-book-2012}
K. Schm\"{u}dgen, \textit{Unbounded Self-adjoint Operators on
Hilbert Space,} Springer GTM {\bf 265}  (2012).

\bibitem{SEbestyen-Stochel}
Z. Sebestyén, J. Stochel, \textit{On suboperators with
codimension one domains}, J. Math. Anal. Appl., \textbf{360/2}
(2009) 391-397.

\bibitem{STO}
J. Stochel, \textit{An asymmetric Putnam-Fuglede theorem for
unbounded operators}, Proc. Amer. Math. Soc., {\bf 129/8} (2001)
2261-2271.

\bibitem{Stochel-Szafraniec-2003-JMSJ-domination-commutativity}
J. Stochel, F. H. Szafraniec, \textit{Domination of unbounded
operators and commutativity}, J. Math. Soc. Japan \textbf{55/2}
(2003), 405-437.

\bibitem{WEI}
J.~Weidmann, Linear operators in Hilbert spaces (translated from the
German by J. Sz\"{u}cs), Srpinger-Verlag, GTM {\bf 68} (1980).

\end{thebibliography}
\end{document}